\documentclass[a4paper,10pt]{article}
\usepackage{amssymb,amsmath, amsthm}
\usepackage[all]{xy}
\usepackage{algorithmic,algorithm} 
\usepackage[OT2,OT1]{fontenc}

\newtheorem{thm}{Theorem}[section]

\newtheorem{prop}[thm]{Proposition}
\newtheorem{lemma}[thm]{Lemma}
\newtheorem{cor}[thm]{Corollary}

\theoremstyle{definition}
\newtheorem{defn}[thm]{Definition}
\theoremstyle{remark}
\newtheorem{rk}[thm]{Remark}
\newtheorem{ex}[thm]{Example}
\newcommand\into{\hookrightarrow}

\newcommand\Q{\mathbb{Q}}
\newcommand\C{\mathbb{C}}
\newcommand\KC{\mathcal{C}}
\newcommand\F{\mathbb{F}}

\newcommand\Z{\mathbb{Z}}
\newcommand\R{\mathbb{R}}

\renewcommand{\div}{\operatorname{div}}
\newcommand\Spec{\mathop{\rm Spec}\nolimits}
\renewcommand\O{\mathcal{O}}

\newcommand\Char{\mathop{\rm char}\nolimits}
\newcommand\ord{\mathop{\rm ord}\nolimits}
\newcommand\supp{\mathop{\rm supp}\nolimits}
\newcommand\Reg{\mathop{\rm Reg}\nolimits}

\newcommand{\Div}{\operatorname{Div}}
\newcommand{\Pic}{\operatorname{Pic}}
\newcommand{\len}{\operatorname{length}}

\renewcommand{\Im}{\operatorname{Im}}
\newcommand{\To}{\longrightarrow}
\newcommand{\BP}{{\mathbb P}}

\newcommand{\surj}{\longrightarrow\!\!\!\!\to}

\usepackage{url}
\title{Computing canonical heights using arithmetic intersection theory}

\author{Jan Steffen M\"uller\thanks{Supported by DFG-grant STO 299/5-1}}

\begin{document}
\maketitle

\section{Introduction}\strut\\[1mm]\label{intro0}
The canonical height $\hat{h}$ on an abelian variety $A$ defined over a global field $k$ is an object of fundamental importance in the study of the arithmetic of $A$. 
For many applications it is required to {\em compute}  $\hat{h}(P)$ for a given point $P\in A(k)$.  
For instance, given generators of a subgroup of the Mordell-Weil group $A(k)$ of finite index, this is necessary for most known approaches to the computation of generators of the Mordell-Weil group $A(k)$. 
Furthermore, the regulator of $A(k)$, which appears in the statement of the conjecture of Birch and Swinnerton-Dyer, is defined in terms of the canonical height and thus we need the ability to compute canonical heights in order to gather numerical evidence for the conjecture in the case of positive rank.  

Here we are concerned with the case where $A$ is the Jacobian variety of a smooth projective curve $C$ of genus $g$ over $k$. 
If $g\le3$, it is known how to compute canonical heights using arithmetic on an explicit
embedding of the Kummer variety $K$ of $A$ into $\BP^{2^g-1}$ (cf. \cite{SilvermanHeights,
FlynnSmart, StollH2, StollG3} and \cite[Chapter~3]{thesis}). 
Trying to imitate this approach in the higher genus case quickly causes problems, as the
Kummer variety becomes rather complicated (see the discussion in 
\cite[Chapter~4]{thesis}).

Instead we propose to use a result due to Faltings \cite{Faltings} and Hriljac
\cite{Hriljac} (see Theorem \ref{fh}), expressing the canonical height in terms of
arithmetic intersection theory. 
Here the non-archimedean intersection multiplicities take
place on regular models of $C$, whereas the archimedean intersection multiplicities are given in terms of Green's functions on the Riemann surface associated to $C$. 
In Section \ref{intro} we discuss the local theory before putting the local results together in Section \ref{Falt-Hril}, culminating in Theorem \ref{fh} which establishes the connection between $\hat{h}$ and arithmetic intersection theory.

In Section \ref{compgns} we show how the necessary arithmetic intersection multiplicities can be computed in practice. 
In the non-archimedean case we reduce the problem to the computation of certain Gr\"obner bases. 
Then we show that the archimedean intersection multiplicities can be computed using theta functions with respect to the complex torus $\C^g/\Lambda$ associated to $A$. 
In order to make these steps practical, we need to be able to decompose divisors into prime divisors and to work on $\C^g/\Lambda$ explicitly.

We present a practical algorithm, implemented in the computer algebra system {\tt Magma} \cite{Magma}, for the computation of $\hat{h}$ for hyperelliptic curves in Section \ref{hyper} by explaining how these two points can be resolved in that case. 
Several examples are given in Section \ref{exs}, where the performance of the algorithm is investigated as well. 
Finally, we elaborate on what is needed to extend the algorithm to non-hyperelliptic curves. 
A different, but similar, algorithm for the computation of $\hat{h}$ using arithmetic intersection theory has been developed independently by Holmes \cite{David}.

{\em Acknowledgments:} The research presented here is also described in Chapters~5 and 6 of my PhD dissertation \cite{thesis}. I would like to thank my advisor Michael Stoll for suggesting this topic and for his constant help and encouragement.

Some of the research described in this work was conducted while I was visiting the University of Warwick and the University of Sydney. It is my pleasure to thank both institutions for their hospitality, as well as David Holmes, Samir Siksek and Steve Donnelly for their invitations.

\section{Local N\'eron symbols}\strut\\[1mm]\label{intro}
In this section we discuss the theory of local N\'eron symbols whose existence was first
proved by N\'eron in \cite{quasi}. We shall present an interpretation that is suitable for
explicit computations, following essentially Gross \cite{Gross} and Hriljac \cite{Hriljac}. The content of the latter work is also discussed by Lang in \cite{LangArak}. In order to present these results, we need to briefly recall some basic notions of intersection theory on arithmetic surfaces.

In the following~3 sections $C$ denotes a smooth projective geometrically connected curve of positive genus $g$, defined over a field $k$ which will be specified as we go along. 
Let $\Div(C)$ denote the group of divisors on $C\times_k k^{\mathrm{sep}}$, where
$k^\mathrm{sep}$ is a separable closure of $k$.
For an extension $k'$ of $k$ contained in $k^\mathrm{sep}$ we denote the subgroup of $k'$-rational divisors by $\Div(C)(k')$. 
For each $n\in\Z$ the set $\Div^n(C)$ is defined to be the set of divisors of degree equal to $n$ and we set
\[\Div^n(C)(k):=\Div^n(C)\cap \Div(C)(k).\]
If $f\in k(C)^*$ and $D=\sum_jm_j(Q_j)\in\Div^0(C)(k)$ is relatively prime to $\div(f)$, then we set
$f(D):=\prod_jf(Q_j)^{m_j}.$

Let $k$ be a non-archimedean local field valuation $v$ with discrete valuation ring
$R$, uniformizing element $\pi$, residue field $\mathfrak{k}$ and spectrum $S=\Spec(R)$.

\begin{defn} A {\em model} $\psi:\KC\to S$ of $C$ over $S$ is an integral, flat,
	Noetherian $S$-scheme of dimension~2 whose generic fiber is isomorphic to $C$. 
\end{defn}
Let $\psi:\KC\to S$ denote a model of $C$ over $S$. 
By abuse of notation, we often omit $\psi$ and simply call $\KC$ a model of $C$ over $S$.
We denote	the special fiber of $\KC$ by $\KC_v$.
Then $\KC_v$ is connected by \cite[Corollary~8.3.6]{LiuBook}.

Let $\Div(\KC)$ denote the group of Weil divisors on $\KC$.
If $D\in\Div(C)(k)$ is prime, then we write $D_{\KC}$ for the Zariski closure of $D$ on $\KC$. This is a prime divisor on $\KC$ and we extend the operation $D\mapsto D_{\KC}$ to all of $\Div(C)(k)$ by linearity. 

We want to use intersection theory on models of $C$ over $S$. 
Although this can be defined more generally, it is convenient to restrict to proper regular models. 
For a proof that such a model always exists, see \cite[\S8.3.4]{LiuBook}.
In our algorithm we shall need a proper regular model of a specific kind; this will be
discussed in Section \ref{RegMods}. 

So suppose that $\psi:\KC\To S$ is a proper regular model of $C$ over $S$. 

\begin{defn}\cite[\S III.2]{LangArak}\label{AsInt}
Let $D,E$ be two effective divisors on $\KC$ without common component and let $P\in\KC_v$ be a closed point. Let $I_{D,P}$ and $I_{E,P}$ be defining ideals of $D$ and $E$, respectively, in the local ring $\O_{\KC,P}$. Then the integer  
\[
 i_P(D,E):=\len_{\O_{\KC,P}}\left(\O_{\KC,P}/(I_{D,P}+I_{E,P})\right)
\]
is called the {\em intersection multiplicity of $D$ and $E$ at $P$}.
The {\em total intersection multiplicity of $D$ and $E$} is
\[
 i_v(D,E):=\sum_Pi_P(D,E)[\mathfrak{k}(P):\mathfrak{k}],
\]
where the sum is over all closed points $P\in \KC_v$.
Finally, we extend $i_P$ and $i_v$ by linearity to divisors $D,E\in\Div(\KC)$ without common component.
\end{defn}

A {\em fibral $\Q$-divisor} is an element of the $\Q$-vector space $\Div_v(\KC)$ generated by the
irreducible components of $\KC_v$. If $D\in\Div^0(C)(k)$, then we denote by $\Phi_{v,\KC}(D)$ a fibral $\Q$-divisor on $\KC$ such
that $D_\KC+\Phi_{v,\KC}(D)$ has trivial intersection with all elements of $\Div_v(\KC)$.
That such a fibral $\Q$-divisor always exists was first proved by Hriljac, cf. \cite[Theorem~III.3.6]{LangArak}.

Now we have assembled all ingredients necessary to define the central objects of this section in the non-archimedean case.
\begin{defn}\label{DefNerSym1}
The {\em local N\'eron symbol on $C$ over $k$} is defined on divisors $D,E\in \Div^0(C)(k)$ with disjoint support by
\[ \langle D, E\rangle_v := i_v(D_{\KC}  + \Phi_{v,\KC}(D) , E_{\KC})\log \#\mathfrak{k}.\] 
\end{defn}
\begin{rk}\cite[Theorem~III.5.2]{LangArak}.
The local N\'eron symbol depends neither on the choice of the regular model $\KC$ nor on
the choice of $\Phi_{v,\KC}(D)$.
\end{rk}

Next we let $k$ denote an archimedean local field.
We can assume $k=\C$ (see Proposition \ref{locneronprops}(d) below), so that $C(k)$ is actually a compact Riemann surface. 
In arithmetic intersection theory one uses Green's functions to define archimedean
intersection multiplicities, but for us a somewhat weaker notion suffices.
The next result follows from \cite[Theorem~13.5.2]{LangFund} combined with
\cite[Proposition~II.1.3]{LangArak}, see also \cite[\S3]{Gross}.

\begin{prop}\label{Green} 
Let $X$ be a compact Riemann surface and let $d\mu$ be a positive volume form on $X$ such that $\int_X d\mu=1$. For each $E\in\Div(X)$ there exists a function 
\[ 
 g_E: X\setminus\supp(E)\to \R,
\]
called an {\em almost-Green's function with respect to $E$ and $d\mu$}, such that the following properties are satisfied:
\begin{itemize}
 \item[(i)] The function $g_E$ is $C^\infty$ outside of $\supp(E)$ and has a logarithmic
	 singularity along $E$.
 \item[(ii)] 
 \[
  \deg(E)d\mu=\frac{i}{\pi}\partial\overline{\partial}g_E.
 \]
\end{itemize}
\end{prop} 

Let $v$ be the absolute value on $k$ and fix a volume form $d\mu$ on $C(k)$ such that
$\int_X d\mu=1$.
\begin{defn}
The pairing $\langle \cdot, \cdot\rangle_v$ that associates to all $D,E\in\Div^0(C)(k)$ with disjoint support the {\em intersection multiplicity} 
\[
i_v(D, E) := g_{E}(D) 
\]
is called the {\em local N\'eron symbol on $C$ over $k$}.
\end{defn}
\begin{rk}
	It follows from \cite[Proposition~II.1.3]{LangArak} that the local N\'eron symbol does
	not depend on the choice of $g_E$ or $d\mu$. 
	See also \cite[Theorem~III.5.3]{LangArak}.
\end{rk}

We list the most important properties of the local N\'eron symbol, both non-archimedean and archimedean, in the following proposition. 
\begin{prop}\label{locneronprops}(N\'eron, Gross, Hriljac) Let $k$ be a local field with
	valuation $v$. The local N\'eron symbol satisfies the following properties, where $D,E\in\Div^0(C)(k)$ have disjoint support.
\begin{itemize}
 \item[(a)] The symbol is bilinear.
 \item[(b)] The symbol is symmetric.
 \item[(c)] If $f\in k(C)^*$, then we have $\langle D, \div(f)\rangle_v=-\log|f(D)|_v$. 
 \item[(d)] If $k'$ is a finite extension of $k$ with valuation $v'$ extending $v$, then we have
	 $\langle D,E\rangle_{v'}=[k':k]\langle D,E\rangle_v$.
\end{itemize} 
\end{prop}
\begin{proof}
	See \cite[\S III.5]{LangArak} and \cite[Theorems~11.3.6, 11.3.7]{LangFund}.

\end{proof}
\section{N\'eron symbols and canonical heights}\strut\\[1mm]\label{Falt-Hril}
In this section we let $k$ denote a global field with ring of integers $\O_k$. We assume
that $C$ is given by $\O_k$-integral equations. Let $M_k=M^0_k\cup M^\infty_k$ denote the
set of places of $k$, with absolute values $|\cdot|_v$ normalized to satisfy the product formula. 
Here $M^0_k$ (respectively $M^\infty_k$) denotes the set of non-archimedean (respectively archimedean)
places.  
For each place $v\in M_k$ we let $k_v$ denote the completion of $k$ at $v$. 
If $v\in M^0_k$, we let $\O_v$ be the ring of integers at $v$.




Let $A$ denote the Jacobian variety of $C$ and let $K$ denote its Kummer variety $A/\{\pm 1\}$. 
Let $K\into\BP^{2^g-1}$ be an embedding of $K$ and let
\[\kappa:A\surj K\into\BP^{2^g-1}.\]
\begin{defn}\cite{quasi},\cite{Gross},\cite{FlynnSmart}
The {\em canonical height (or N\'eron-Tate height) on $A$} is the function defined by
\[
\hat{h}(P):=\lim_{n\to\infty}4^{-n}h(\kappa(2^nP)),
\]
where $h$ is the usual absolute height on $\BP^{2^g-1}$. The {\em canonical height pairing (or N\'eron-Tate height pairing) on $A$} is defined by
\[
(P,Q)_{\mathrm{NT}}:=\frac12(\hat{h}(P+Q)-\hat{h}(P)-\hat{h}(Q)).
\]
\end{defn}
Note that taking the absolute height on $\BP^{2^g-1}$ means that $\hat{h}$ does not depend
on $k$.

Next we shall relate the canonical height to N\'eron symbols.
If $D\in \Div(C)(k)$ and $v\in M_k$, then we define $D_v:=D\otimes_k k_v$.
For $D,E\in\Div^0(C)(k)$ with disjoint support and $v\in  M_k$ we define $\langle
D,E\rangle_v:=\langle D_v,E_v\rangle_v$ and the {\em global N\'eron symbol} by
\[
\langle D,E\rangle := \sum_{v\in M_k} \langle D,E\rangle_v.
\]
This is a finite sum, since over all places of good reduction the Zariski
closure $\overline{C}^v$ of the given equations of $C\times_k k_v$ over $\Spec(\O_v)$ is a proper regular model over
$\Spec(\O_v)$. 
Hence we have $\Phi_{v,\overline{C}^v}(D_v)=0$ for all such $v$ and
$i_v(D_{v,\overline{C}^v},E_{v,\overline{C}^v})\ne0$ for only finitely many such $v$.

By Proposition \ref{locneronprops}(c) and the product formula, $\langle D,E\rangle$ only depends
on the classes $[D], [E]\in \Pic^0(C)$ and hence we can drop the assumption that $D$ and
$E$ have disjoint support.

\begin{thm}\label{fh}(Faltings \cite{Faltings}, Hriljac \cite{Hriljac})
	Suppose $C$ is a smooth projective geometrically connected curve of positive genus $g$ defined over a global field $k$.
	If $D,E \in \Div^0(C)(k)$, then we have 
	\[\langle D, E\rangle =  -([D],[E])_{\mathrm{NT}}.\] 
	\end{thm}

The practical importance of this result lies in the fact that we can, at least in principle, compute the canonical height on the Jacobian using data associated to the curve. 
We do not impose any further conditions on $C$ (yet). Suppose that we are given a point $P\in A(k)$ and we want to compute its canonical height $\hat{h}(P)$. In order to use Theorem \ref{fh} for this purpose, we proceed as follows:
\begin{itemize}
 \item[(1)] Find divisors $D,E\in\Div^0(C)(k)$ such that $[D]=[E]=P$ and $\supp(D)\cap\supp(E)=\emptyset$. 
 \item[(2)] Determine a finite set $U$ of places $v\in M^0_k$ such that $\{v\in M_k^0:
 \langle D,E\rangle_v \ne0\}\subset U$.
 \item[(3)]	Find a proper regular model $\KC$ of $C\otimes_k k_v$ over $\Spec(\O_v)$ for all $v\in U$.
 \item[(4)] Compute  $i_v(D_{v,\KC},E_{v,\KC})$ for all $v\in U$.
 \item[(5)] Compute $i_v(\Phi_{v,\KC}(D_{v,\KC}),E_{v,\KC})$ for all $v\in U$ of bad
	 reduction. We call this the {\em correction term}.
 \item[(6)] Find an almost-Green's function $g_{E_v}$ and compute $g_{E_v}(D_v)$ for all $v\in M^\infty_k$.
 \item[(7)] Sum up all local N\'eron symbols.
\end{itemize}
We deal with these steps in the following sections. 
\begin{rk}
We shall tacitly assume from now on that step (1) is always possible in principle, that is every $P$ we encounter can be represented using a $k$-rational divisor. According to \cite[Proposition~3.3]{PS} this is guaranteed whenever the curve has a $k_v$-rational divisor of degree~1 for all $v\in M_k$. If we have $P\in A(k)$ which cannot be represented using a $k$-rational divisor, then we have two options:
\begin{itemize}
	\item[1)] Work over a field extension $k'$ of $k$ such that there exists some $D\in\Div^0(C)(k')$ satisfying $[D]=P$.
	\item[2)] Compute a multiple $n P$ such that there exists $D\in\Div^0(C)(k)$ satisfying $[D]=n P$ and use the quadraticity of the canonical height. 
\end{itemize}
The existence of $n$ as in 2) follows from \cite[Proposition 3.2]{PS}; we can take for $n$ the period of $C$ over $k$.
\end{rk}
\section{Computing N\'eron symbols}\strut\\[1mm]\label{compgns}
In this section we shall address the steps needed for the computation of global N\'eron symbols introduced in the previous section. The first two steps are global in nature and can be viewed as preparatory steps for the remaining four sections which are local.

\subsection{Finding suitable divisors of degree zero}\strut\\[1mm]\label{Move1}
The basic reference for large parts of the remainder of this section is \cite{Hess}. If an ideal $I$ is generated by elements $b_1,\ldots,b_n$, then we write $I=(b_1,\ldots,b_n)$.
Let $k$ be an arbitrary field. There are essentially two ways to represent a divisor $D\in\Div(C)(k)$. 
\begin{itemize}
	\item[(a)] As a sum \[
	D=\sum_im_iD_i,
	\]
	where $D_i\in\Div(C)(k)$ is irreducible over $k$ and $m_i\in\Z$ for all $i$. We call this the {\em free representation of $D$}.
	\item[(b)] Assuming $D$ is effective, using a defining ideal 
	\[
		I_D\subset k[C^a],
	\]
	where $C^a$ is an affine chart of $C$ containing $D$. We call this an {\em ideal representation of $D$}. 
\end{itemize}

Since in our intended applications we are allowed (and often even required) to vary divisors in their linear equivalence classes, it is a natural question to ask whether it is possible to find divisors linearly equivalent to a given divisor in a way that facilitates explicit computations.
\begin{lemma}(Hess)\label{RedEx}
For all $D\in\Div(C)(k)$ and effective $A\in\Div(C)(k)$ there exists an effectively computable triple $(\tilde{D},r,a)$, where $\tilde{D}\in\Div(C)(k)$ is effective, $r\in\Z$ and $a\in k(C)$ such that $\deg(\tilde{D})<g+\deg(A)$ and we have
\[
 D=\tilde{D}+rA+\div(a).
\]We call $\tilde{D}$ a {\em reduction} of $D$ along $A$.
If $\deg(A)=1$, then $\tilde{D}$ is the unique effective divisor such that $\dim(\mathcal{L}(\tilde{D}-r'A))=0$ for all $r'\ge1$.
In this case we have $D\sim\tilde{D}+rA$, where $r\in\Z$ is the maximal integer such that $\dim(\mathcal{L}(D-rA))=1$.
\end{lemma}
\begin{proof}
See \cite[\S8]{Hess}. 
\end{proof}

Now assume that $k$ is a global field, that we are given some divisor $D\in \Div^0(C)(k)$ and we want to find $E\sim D$ such that $E$ and $D$ have disjoint support. In other words, we are looking for an effective version of the moving lemma. However, we would like to keep the computations as simple as possible and this means that we would like to work with divisors that are reduced along some effective divisor of small degree whenever possible. 

This leads to the following method:
\begin{itemize}
 \item[1.] Pick two effective divisors $A,A'\in\Div(C)(k)$ with disjoint support.
 \item[2.] Compute multiples $nD$, where $n=1,-1,2,-2,\ldots$ and reduce them along $A$ and $A'$ until we find some $n$ and $n'$ such that the reduction $\tilde{D}_n$ of $nD$ along $A$ and the reduction $\tilde{D}_{n'}$ of $n'D$ along $A'$ have disjoint support.
 \item[3.] Let $r_n,r_{n'}\in\Z$ such that $nD\sim\tilde{D}_n+r_nA$ and $n'D\sim\tilde{D'}_{n'}+r_{n'}A'$. Compute
 \begin{align*}
 \langle D,D\rangle&=\frac{1}{nn'}\langle \tilde{D}_n+r_nA,\tilde{D'}_{n'}+r_{n'}A' \rangle	\\
 &= \frac{1}{nn'}\langle\tilde{D}_n,\tilde{D'}_{n'} \rangle+\frac{r_n}{nn'}\langle A, \tilde{D'}_{n'} \rangle+\frac{r_{n'}}{nn'}\langle \tilde{D}_n, A'\rangle+\frac{r_nr_{n'}}{nn'}\langle A,A' \rangle.
 \end{align*}
\end{itemize}
In practice integers $n,n'$ of fairly small absolute value usually suffice. 
\begin{rk}
In the method above, it is not obvious how to pick $A$ and
$A'$ in a way that facilitates explicit computations.  
If we have $k$-rational divisors $A,A'$ of degree~1 on $C$
then they can be used. 
If $C$ is a plane curve, then we can use the zero or pole divisors of functions of the
form $x-\zeta$, where $\zeta\in k$. 
See Section \ref{Move1h} for the case of hyperelliptic curves.
In general the choice of $A$ and of $A'$ depends on the specific situation.
\end{rk}

\subsection{Determining relevant non-archimedean places}\strut\\[1mm]\label{factoring}
We continue to let $k$ denote a global field. Given two divisors $D$ and $E$ with disjoint
support, we have to find a finite set of places $v\in M^0_k$ such that any non-archimedean
$v$ satisfying $\langle D,E\rangle_v\ne0$ must lie in $U$.

We can assume that $D$ and $E$ are effective and use their respective ideal representations. The idea is to cover our curve by affine patches $C^1,\ldots,C^s$ and determine the relevant places for each patch using Gr\"obner bases. We refer to \cite[Chapter~4]{ALGB} for an introduction to the theory and applications of Gr\"obner bases for polynomial rings over Euclidean rings. 

So let 
\[C^i=\Spec k[x_1,\ldots,x_n]/(G_{i,1}(x_1,\ldots,x_n),\ldots,G_{i,m_i}(x_1,\ldots,x_n))
\]
be such an affine patch, where $G_{i,j}(x_1,\ldots,x_n)\in\O_k[x_1,\ldots,x_n]$ for all $j$. 
From now on we will assume that for each $v\in M^0_k$ there is some $i,j$ such that
$G_{i,j}$ is a $v$-adic unit. 
Note that this implies that the Zariski closure $\overline{C}^v$ of $C\times_k k_v$ over
$\Spec(\O_v)$ is a model for $C$ over $\Spec(\O_v)$.

Suppose for now that the ring of integers $\O_k$ is Euclidean and that $D$ and $E$ are represented by ideals $I_{D,i}$ and $I_{E,i}$, respectively, on $C^i$ for each $i$. 
In fact we can assume that $I_{D,i}$ and $I_{E,i}$ are given by bases whose elements are in $\O_k[x_1,\ldots,x_n]$. If we compute a Gr\"obner basis $B_i$ of
\[
 I_{D,E,i}:=(G_{i,1}(x_1,\ldots,x_n),\ldots,G_{i,m_i}(x_1,\ldots,x_n))+I_{D,i}+I_{E,i}
\] 
over $\O_k$, then $B_i$ contains a unique element $q_{D,E,i}\in \O_k$. 

We define the set $U$ by
\[
	U:=\{v\in M^0_k: \ord_v(q_{D,E,i})>0\textrm{ for some }i\}.
\]
For the proof of the following Lemma, we need the notion and existence of a
desingularization in the strong sense of $\overline{C}^v$ for each $v$. 
The necessary details are presented in Section \ref{RegMods} below.
\begin{lemma}
Any non-archimedean place $v$ such that $\langle D,E\rangle_v>0$ is contained in $U$.	
\end{lemma}
\begin{proof}
	Suppose that $\langle D,E\rangle_v>0$ and let $\xi:\KC\to\overline{C}^v$ denote a desingularization of
	$\overline{C}^v$ in the strong sense.
The existence of $\xi$ is asserted by \ref{DesingExists}, since by assumption
$\overline{C}^v$ is a model of $C\times_k k_v$ over $\Spec(\O_v)$.
Then we must have
\begin{equation}\label{pgtf}
	i_v(\Phi_{v,\KC}(D),E_{v,\KC})>0\textrm{  and  }i_v(D_{v,\KC},\Phi_{v,\KC}(E))>0
\end{equation}
or
\begin{equation}\label{dgtf}
i_v(D_{v,\KC},E_{v,\KC})>0.
\end{equation}

If \eqref{pgtf} holds, then $v$ must be a place of bad reduction such that both
$D_{v,\overline{C}^v}$ and
$E_{v,\overline{C}^v}$ intersect the singular locus of $\overline{C}^v$, since otherwise
either $\Phi_{v,\KC}(D)$ or $\Phi_{v,\KC}(E)$ vanish.

If \eqref{dgtf} holds, the fact that $\xi$ is an isomorphism outside the singular locus
of $\overline{C}^v$ implies that the closures $D_{v,\overline{C}^v}$ and
$E_{v,\overline{C}^v}$ do not have disjoint supports.
But this means that there is a point in the support of $D$ and a point in the support of
$E$ having the same reduction modulo $v$. The claim follows easily from this.
\end{proof}
Hence the problem of determining $U$ comes down to a combination of computing Gr\"obner
bases and factoring.

If $\O_k$ is not a Euclidean ring, then we can still use this Gr\"obner basis approach by writing $k$ as $k'(\alpha)$ for a primitive element $\alpha$ of $k$ over $k'$, where $k'=\Q$ if $k$ is a number field and $k'=\F_p(T)$ if $\Char(k)=p\ne 0$. This trick appears in \cite[Exercise~4.3.1]{ALGB}. We add a new variable $t$ to $\O_{k'}[x_1,\ldots,x_n]$, satisfying the relation
\[
 \phi_{\alpha}(t)=0,
\] 
where $\phi_{\alpha}$ is the minimal polynomial of $\alpha$ over $k'$, and replace any occurrence of $\alpha$ in $I_{D,E,i}$ by $t$. Now we get at most one $q_{D,E,i}(t)\in\O_{k'}[t]\setminus\O_{k'}$ in the Gr\"obner basis of $I_{D,E,i}$, but we might also have some $q'_{D,E,i}\in\O_{k'}$. We factor the principal ideal $(q_{D,E,i}(\alpha))$ in $\O_k$ and, if necessary, the principal ideal $(q'_{D,E,i})$ in $\O_k$ to find the relevant $v\in M^0_k$. 

\subsection{Computing regular models}\strut\\[1mm]\label{RegMods}
In the following three sections we let $k$ denote a non-archimedean local field
with valuation $v$. 
Let $R$ be its discrete valuation ring with spectrum $S=\Spec(R)$, uniformizing element $\pi$
and residue field $\mathfrak{k}$. 
Suppose that $C$ is given by $R$-integral equations and that the Zariski closure
$\overline{C}$ of $C$ over $S$ is a model of $C$ over $S$. 
In practice, it is easy to find such equations for $C$ (for instance by requiring that
there is some unit amongst their coefficients).

\begin{defn}
Let $\KC'$ be a model of $C$ over $S$. A {\em desingularization of $\KC'$ in the strong
sense} is a proper birational morphism $\xi:\KC\to\KC'$ such that $\mathcal{C}$ is
regular and $\xi$ is an isomorphism above every regular point of $\KC'$.
\end{defn}
The proof of the following result can be found in \cite[Corollary~8.3.51]{LiuBook}. 
It extends a proof due to Lipman, recalled in \cite{Artin}, of the same result in the
special case that the given model of $C$ over $S$ is excellent.
\begin{lemma}\label{DesingExists}
Let $\KC'$ be a model of $C$ over $S$. Then there exists a desingularization of $\KC'$ in the strong
sense.
\end{lemma}
Although the theory works for any proper regular model of $C$ over $S$, for our algorithm
we need a desingularization $\KC\to\overline{C}$ in the strong sense, as in the previous section.

Lipman's proof is effective:
The idea is to normalize $\overline{C}$ and then blow up the resulting model along its
(necessarily isolated) irregular points.
Repeating this process yields a desingularization of $\overline{C}$ in the strong sense after
finitely many steps. 
The main problem is that normalizations are more difficult from a computational point of view than blow-ups.

A different effective proof of the existence of a desingularization of $\overline{C}$ in the strong sense is
given by Cossart, Jannsen and Saito in \cite{Jannsen}, although they assume that $\overline{C}$ is excellent.
The advantage of their method is that it only involves blow-ups,
either along isolated irregular points or along irreducible components if the singular locus
of the respective model has positive dimension.


This method has been implemented in {\tt Magma} by Donnelly. 
The data that can be accessed once $\KC$ has been constructed using {\tt Magma} includes the blow-up maps on enough affine patches to cover all intermediate models, the intersection matrix of $\KC_v$ and the multiplicities of the irreducible components.

A subtle point is that in the proof in \cite{Jannsen} the blow-ups have to be
performed in a specific order and this has not yet been included in the implementation,
but it should be possible without too much difficulty. 
In any case, the implementation works well in practice. 
The only essential restrictions at the moment are that the curve has to be planar and that blow-ups along components are not
implemented unless $C$ is defined over $\Q$.

\subsection{Computing non-archimedean intersection multiplicities}\strut\\[1mm]\label{naint}
We keep the notation from the previous section and fix a desingularization
$\xi:\KC\to\overline{C}$ in the strong sense, covered by affine patches 
\[
\KC^i=\Spec R[x_1,\ldots,x_{s_i}]/(H_{i,1}(x_1,\ldots,x_{s_i}),\ldots,H_{i,t_i}(x_1,\ldots,x_{s_i})).
\]
For computational purposes we shall assume for the moment that we have two effective divisors $D$ and $E$ with disjoint support whose closures $D_\KC$ and $E_\KC$ lie entirely in an affine piece $\KC^i$. 

The following lemma is a well-known result from commutative algebra saying that quotients and localizations commute.
\begin{lemma}\label{Matsu}
Let $A$ be a commutative ring with unity and let $T\subset A$ be a multiplicative subset. Let $I\subset A$ be an ideal and let $\bar{T}$ denote the image of $T$ in $A/I$. Then we have
\[
A_T/IA_T\cong (A/IA)_{\bar{T}},
\]
where the subscripts denote localizations.
\end{lemma}
\begin{proof}
See  \cite[Theorem~4.2]{Matsum}.
\end{proof}

Let $I_{D,i}$ and $I_{E,i}$ denote defining ideals of $D_\KC$ and $E_\KC$ in the ring $\O_{\KC^i}$, respectively. For the computation of the intersection multiplicity we use the following version of the Chinese remainder theorem for modules. 
\begin{prop}\label{crt}
 Let $A$ be a commutative ring and let $M$ be an Artinian and Noetherian $A$-module. Then there is an isomorphism of $A$-modules
 \[
  M\cong \bigoplus_P M_P,
 \]
 where the sum is over all maximal ideals $P$ of $A$ and $M_P$ denotes the localization of $M$ at $P$.
\end{prop}
\begin{proof}
See \cite[Theorem~2.13]{Eisenbud}.
\end{proof}

\begin{prop}\label{IntForm}
Suppose that $D_\KC\cap E_\KC$ only intersects a single component $\Gamma$ of $\KC^i_v$.
Let $\O_\Gamma$ denote the ring of regular functions on $\Gamma$. Then we have
 \[
  i_v(D_{\KC},E_{\KC})=\len_{\O_\Gamma}\left(\O_\Gamma/(I_{D,i}+I_{E,i})\O_\Gamma\right)
 \]
\end{prop}
\begin{proof}
From Proposition \ref{crt} we get an isomorphism of $\O_\Gamma$-modules
\begin{equation}\label{ModsIsom}
 \O_\Gamma/(I_{D,i}+I_{E,i})\cong \bigoplus_P\O_{\KC^{i},P}/(I_{D,i}+I_{E,i}),
\end{equation}
where the sum is over all maximal ideals of $\O_\Gamma$, that is, over all closed points $P\in\Gamma$.
 By our assumptions we have
 \begin{align*}
  i_v(D_{\KC},E_{\KC})=&\sum_Pi_P(D_{\KC},E_{\KC})[\mathfrak{k}(P):\mathfrak{k}]\\
  						 =&\sum_P\len_{\O_{\KC^{i},P}}\left(\O_{\KC^{i},P}/(I_{D,i}+I_{E,i})\right)[\mathfrak{k}(P):\mathfrak{k}]\\
  						 =&\sum_P\len_{\O_\Gamma}\left(\O_{\KC^{i},P}/(I_{D,i}+I_{E,i})\right)\\
  						 =&\len_{\O_\Gamma}\bigoplus_P\left(\O_{\KC^{i},P}/(I_{D,i}+I_{E,i})\right)\\
  						 =&\len_{\O_\Gamma}\left(\O_\Gamma/(I_{D,i}+I_{E,i})\right)
 \end{align*}
 using \eqref{ModsIsom}, additivity of the length and the fact that if $M$ is an $\O_\Gamma$-module that is also an $\O_{\KC^{i},P}$-module for some closed point $P\in\Gamma$, then we have
 \[
  \len_{\O_\Gamma}(M)=\len_{\O_{\KC^{i},P}}(M)[\mathfrak{k}(P):\mathfrak{k}].
 \]
\end{proof}
Instead of computing $\len_{\O_\Gamma}\left(\O_\Gamma/(I_{D,i}+I_{E,i})\O_\Gamma\right)$ for each component $\Gamma$ of $\KC^i_v$, we can proceed more directly. Let
\begin{equation}\label{ADEv}
 A_{D,E,i,v}:=\left(R[x_1,\ldots,x_{s_i}]/I_{D,E,i,v}\right)_{(\pi)}
\end{equation}
where
\begin{equation}\label{IDEv}
 I_{D,E,i,v}=(H_1(x_1,\ldots,x_{s_i})\ldots,H_{t_i}(x_1,\ldots,x_{s_i}))+I_{D,i}+I_{E,i}.
\end{equation}
\begin{cor}
We have
\[
 i_v(D_{\KC},E_{\KC})=\len_{\O_{\KC^{i}_v}}A_{D,E,i,v}
\]
\end{cor}
\begin{proof}
Use Proposition \ref{IntForm} and additivity of the length.
\end{proof}

Computing $\len_{\O_{\KC^{i}_v}}A_{D,E,i,v}$ is rather easy and can be done, for instance, in {\tt Magma}. 
The crucial step is the computation of a Gr\"obner basis $B$ of $I_{D,E,i,v}$ over the Euclidean ring $R$.
Here the property of $B$ that we need is that for every $h\in R[x_1,\ldots,x_{s_i}]$ multivariate division
of $h$ by $B$ yields a unique remainder that we call $h\bmod B$.

The idea is to count residue classes as follows, where we start with $d=0$.

Suppose $d\ge 0$. For each monomial $g$ of total degree $d=\deg(g)$, find the integer $n$ such that
$\pi^jg$ is new for $j\in\{0,\ldots,n-1\}$,
where we call $\pi^jg$ {\em new} if the residue classes of $\pi^j g$ has not already been
counted.
We can test the latter by checking whether the total degree of $\pi^jh\bmod B$ is at most
equal to the total degree of $h$ or whether $\pi^j h$ does not divide $h$.

If we find that no monomial of total degree $d$ contributes a new residue
class, then we are done, otherwise we increment $d$ by~1 and repeat this process.
See Algorithm \ref{alg3}. 

\algsetup{indent=2.0 em}
\begin{algorithm*}
\caption{ Computation of $\len_{\O_{\KC^{i}_v}}A_{D,E,i,v}$}\label{alg3}
\begin{algorithmic}
  \STATE $B=\{g_1(x_1,\ldots,x_{s_i}),\ldots,g_r(x_1,\ldots,x_{s_i}),q\}\leftarrow\text{  Gr\"obner basis  of }I_{D,E,i,v}$ 
  \STATE  $m\leftarrow \ord_v(q)$ \quad // $q$ yields $m$ distinct residue classes.
	\STATE $d\leftarrow 0$ \quad // Total degree
	\STATE $T\leftarrow\emptyset$ \quad // Monomials all of whose multiples are known not to be new
  \REPEAT
\STATE $d\leftarrow d+1$ \quad // Increment total degree
  \STATE $V\leftarrow \{g=\prod^{s_i}_{i=1}x_i^{k_i}: k_i\in\mathbb{N},\; \sum^{s_i}_{i=1}k_i=d\text{  and
	}h\nmid g\textrm{  for all }h \in T\}$ \\ // No
	monomial of total degree $d$ outside $V$ can be new.
	  \STATE $m'\leftarrow m$ 
  \FOR{$g\in V$}
    \STATE $n\leftarrow 0$
	\WHILE{$\deg(\pi^ng\bmod B)> d\text{  or  }g\mid \pi^ng\bmod B$}
	  \STATE $n\leftarrow n+1$ \quad // $\pi^ng$ is new.
	  \ENDWHILE
      \STATE $m\leftarrow n+m$ \quad // Get $n$ new residue classes.
	  \IF{$n=0$}\STATE $T\leftarrow T\cup\{g\}$ \quad // No multiple of $g$ is new. \ENDIF 
	  \ENDFOR
	  \UNTIL{$m=m'$} \quad // No monomial of total degree $d$ is new.
\RETURN $m$
\end{algorithmic}
\end{algorithm*}

In order to apply the results of this section we need to be able to find
\begin{itemize}
\item[(a)] an affine cover $\KC_v^i$ of $\KC$ containing $D_\KC\cap E_\KC$,
\item[(b)] ideal representations $I_{D,i}$ and $I_{E,i}$ of $D_\KC|_{\KC^i_v}$ and $E_\KC|_{\KC^i_v}$.
\end{itemize}
If the Zariski closure $\overline{C}$ of $C$ over $\Spec(R)$ is regular, then we can solve (a) and
(b) by decomposing $D$ and $E$ into prime divisors over $k$. 
If the regular model has a more complicated structure, this may not be sufficient and we may have to decompose $D$
and $E$ into prime divisors over the maximal unramified extension $k^\mathrm{nr}$ 
since each such prime divisor reduces to a single point on the special fiber of both
$\overline{C}$ and $\KC$. 
This might be necessary because our strategy is to start with $R$-integral ideal representations of $D$ and $E$ and recursively lift these through the blow-up process. 

Note that we do not actually have to work over $k^\mathrm{nr}$. 
In order to decompose $D$ into prime divisors over $k^\mathrm{nr}$ it suffices to consider
the maximal unramified extension $l/k$ contained in the smallest extension $k(D)/k$ such
that $D$ becomes pointwise rational. 
It is possible to compute with such extensions in {\tt Magma}.

\begin{rk}
The strategy employed to decompose $D$ and $E$ depends on the curve at hand; for hyperelliptic curves there is a straightforward method to decompose divisors that only uses factorization of univariate polynomials as explained in Section \ref{hyper}.
\end{rk}

All of the above is trivial if $D$ and $E$ are pointwise $k$-rational: 
\begin{cor}\label{NaiveInt}
Suppose $D=\sum_ln_l (P_l)$ and $E=\sum_jm_j(Q_j)$, where $n_l,m_j\in\Z$ and all $P_l$ and $Q_j$ are $k$-rational such that $D_\KC\cap\KC_v$ and $E_\KC\cap\KC_v$ contain no singular points of $\KC_v$. Then we have
\[i_v(D_{\KC},E_{\KC})=\sum_{l,j}n_lm_j\min\left\{\ord_v(x_1(P_l)-x_1(Q_j)),\ldots,\ord_v(x_{s_i}(P_l)-x_{s_i}(Q_j))\right\}\]
where $P_l=(x_1(P_l),\ldots,x_{s_i}(P_l)),\;Q_j=(x_1(Q_j),\ldots,x_{s_i}(Q_j))\in C^i$.
\end{cor}
\begin{proof}
This follows easily from Proposition \ref{IntForm}
\end{proof}
See \cite{David} for a similar version of Lemma \ref{NaiveInt}, found independently by
Holmes. 

\begin{rk}\label{NoFac}
A different strategy was brought to the attention of the author by Florian Hess and
consists in computing the intersection multiplicities of $D_\KC$ and $E_\KC$ on all affine
patches $\KC^i$ such that $D_\KC$ and $E_\KC$ intersect on $\KC^i_v$. For simplicity, we
assume that these are $\KC^1$ and $\KC^2$. We then need to subtract from the result all
intersections that take place on both $\KC^1$ and $\KC^2$ which can be expressed as the
length of a certain module. This approach is outlined in \cite{HessStudent}, but it is not
clear how it can be made practical if $\overline{C}\ne\KC$.
\end{rk}

\subsection{Computing the correction term}\strut\\[1mm]\label{CompCorrTerm}
We continue to let $\KC$ denote a desingularization in the strong sense of $\overline{C}$
over $S$.
Suppose that the special fiber $\KC_v$ is equal to $\sum^r_{i=0}n_i\Gamma_v^i$, where
$\Gamma^0_v,\ldots,\Gamma^r_v$ are the irreducible components of $\KC_v$. Let $M_v$ denote the
intersection matrix $\left(i_v(n_i\Gamma^i_v,n_j\Gamma^j_v)\right)_{0\le i,j\le r}$ of
$\KC_v$. 

Suppose we are given a divisor $D\in\Div^0(C)(k)$ and we want to compute a fibral
$\Q$-divisor $\Phi_{v,\KC}(D)=\sum^r_{i=0}\alpha_in_i\Gamma_v^i$ having trivial
intersection with $\Div_v(\KC)$.
Also suppose that we have found both $M_v$ and $s(D)$, where 
\begin{equation}\label{sD}
s(D)=\left(n_0i_v(D_{\KC},\Gamma_v^0),\ldots,n_ri_v(D_{\KC},\Gamma_v^r)\right)^T.\end{equation}
We mention two possible methods here, both easily checked to be correct using \cite[\S
III.3]{LangArak}.
 \begin{itemize}
  \item[(i)] Let $M^+_v$ be the Moore-Penrose pseudoinverse of $M_v$. Then we can set \[
 (\alpha_0,\ldots,\alpha_r)^T:=-M^+_v\cdot s(D).
\]
  \item[(ii)] (Cox-Zucker \cite{cz}) Suppose that there exists some $i$ such that $n_i=1$, say $n_0=1$, and let $M'_v$ be the matrix obtained by deleting the first column and row from $M_v$. We pick $\alpha_0:=0$ and \[
 (\alpha_1,\ldots,\alpha_r)^T:=-M_v'^{-1}\cdot s'(D),
\]
 where $s'(D)$ is the vector obtained by removing the first entry of $s(D)$. 
 \end{itemize}
 
 We can now compute $i_v(\Phi_{v,\KC}(D),E_{\KC})$ easily for $E\in\Div^0(C)(k)$ having support disjoint from $D$. This is simply equal to
\[
s(E)^T\cdot (\alpha_0,\ldots,\alpha_r)^T,
\]
where $s(E)$ is defined as in \eqref{sD}.

It only remains to discuss how $s(D)$ and $s(E)$ can be computed, because $M_v$ can be
computed using {\tt Magma} once $\KC$ has been constructed. 
But this is essentially contained in the previous section:
We decompose $D$ and $E$ into prime divisors
over $k^{\mathrm{nr}}$ and then determine which components the corresponding points map to
by lifting ideal representations recursively through the blow-up process. 
Since for any one of these divisors $P$ all points in its support reduce to the same point in each step, it is easy to pick suitable affine patches covering the intermediate models until we find an affine patch of $\KC$ containing $P_\KC$. 
The final task is the computation of $i_v(P_{\KC},\Gamma^i_v)$ for all components $\Gamma_v^i$ intersecting this affine patch which is easy under our assumptions.

\subsection{Computing archimedean intersection multiplicities}\strut\\[1mm]\label{ArchimInt}
In order to deal with the computation of archimedean local N\'eron symbols it suffices to
consider $k=\C$. Let $C(\C)$ denote the Riemann surface associated to $C$. According to
Section \ref{Falt-Hril} we need to find an almost-Green's function with respect to a
divisor $E\in\Div^0(C)(\C)$. Notice that we can write any such divisor in the form
$E=E_1-E_2$, where $E_1$ and $E_2$ are {\em non-special}, that is they are effective of
degree $g$ and their $\mathcal{L}$-spaces have dimension~1. 
It suffices to determine almost-Green's functions  with respect to non-special divisors and any fixed normalized volume form on $C(\C)$. 

In order to do this it turns out to be useful to work on the analytic Jacobian, which we view as an abelian variety over the complex numbers. Let $\tau$ be an element of the Siegel space $\mathfrak{h}_g$ such that $A(\C)$ is isomorphic to the complex torus $\C^g/\Lambda$, where $\Lambda=\Z^g\oplus\tau\Z^g$. Let the map $j$ be defined by
\[\xymatrix{
 j:\C^g\ar@{->>}[r]&\C^g/\Lambda\ar[r]^{\cong} &A(\C).}
\]
Moreover, we fix an Abel-Jacobi map, that is an embedding $\iota$ of $C(\C)$ into $A(\C)$, and let $\Theta\in\Div(A)$ denote the theta-divisor with respect to $\iota$. Let $S:\Div(C)\To A$ denote the summation map associated to $\iota$. 

On $A(\C)$ we can find the following canonical 2-form: Let $\eta_1,\ldots,\eta_g$ be an orthonormal basis of the differentials of first kind on the Jacobian. Then the canonical 2-form is given by
\[
 \frac{1}{2g}(\eta_1\wedge\bar{\eta}_1+\ldots+\eta_g\wedge\bar{\eta}_g)
\]
and we define the {\em canonical volume form $d\mu$ on $C(\C)$} by pulling this form back using $\iota$, see \cite[\S13.2]{LangFund}. The details are not important for us as the dependence on $d\mu$ for divisors of degree zero. 

For the next theorem, conjectured by Arakelov and proved by Hriljac, we need the concept of N\'eron functions with respect to divisors on $A$, for which we refer to \cite[Chapter~11]{LangFund}. We will introduce a specific N\'eron function in the situation we are interested in shortly. We use the notation $E_P$ to denote the translation of a divisor $E\in\Div(A)$ by a point $P\in A$.
\begin{thm}\label{ArakHril}(Hriljac) Let $E\in \Div^g(C)$ be non-special, let $P=S(E)$ and $E'=([-1]^*(\Theta))_P$. Let $\lambda_{E'}$ be a N\'eron function with respect to $E'$. Then $\lambda_{E'}\circ\iota$ is an almost-Green's function with respect to $E$ and $d\mu$, where $d\mu$ is the canonical volume form on $C(\C)$.
\end{thm}
\begin{proof}
See \cite[Theorem~13.5.2]{LangFund}.
\end{proof}
The great news is that it is not difficult to find N\'eron functions with respect to $\Theta$; we show below that this suffices for our purposes. 

\begin{defn}
Let $g\ge1$ and $a,b\in\Q^g$. Let the function $\theta_{a,b}$ on $\C^g\times \mathfrak{h}_g$ be given by
\[
 \theta_{a,b}(z,\tau)=\sum_{m\in\Z^g}\exp\left(2\pi i\left(\frac{1}{2}(m+a)^T\tau(m+a)+(m+a)^T(z+b)\right) \right).
\]
We call $\theta_{a,b}$ the {\em theta function with characteristic $[a;b]$}.
\end{defn}
Now let $a=(1/2,\ldots,1/2),b=(g/2,(g-1)/2\ldots,1,1/2)\in\Q^g$ and consider $\theta_{a,b}(z):=\theta_{a,b}(z,\tau)$ as a function on $\C^g$. 

\begin{prop}\label{neronpazuki2}(Pazuki)
The function $\theta_{a,b}$ has divisor $j^{*}(\Theta)$. Moreover, the following function is a N\'eron function associated with $\Theta$:
\[
 \lambda_{\Theta}(P)=-\log|\theta_{a,b}(z(P))|+\pi \Im(z(P)) ^T (\Im(\tau))^{-1} \Im(z(P)),
\]
where $j(z(P))=P$.
\end{prop}
\begin{proof}
This was stated without proof by Pazuki in \cite[Proposition~3.2.11]{PazukiThesis}, but it
is in fact rather easy to verify: It is a classical theorem (see \cite[Theorem~13.4.1]{LangFund}) that the divisor of the Riemann theta function
$\theta=\theta_{0,0}$ is a translate by a point $w$ of $\Theta$ and that $2w$ is the image
on $A$ of the canonical class on $C$. Using this it is not hard to see that the odd
function $\theta_{a,b}$ has divisor $j^{*}(\Theta)$. Then one uses \cite[Theorem~13.1.1]{LangFund} to find an expression of a N\'eron function in terms of the normalized theta function \[\theta'_{a,b}(z):=\theta_{a,b}(z)\exp\left(\frac{\pi}{2}z^T(\Im \tau)^{-1}z\right);\] the right hand side in Proposition \ref{neronpazuki2} is equal to this expression after a straightforward manipulation.
 
Alternatively one can show directly that $\lambda_{\Theta}$ satisfies the properties of a N\'eron function. 
\end{proof}
Now suppose that $E=E_1-E_2$, where $E_1,E_2\in\Div(C)$ are non-special divisors with disjoint support, and let $D_1=\sum^d_{i=1} (P_i)$ and $D_2=\sum^d_{i=1} (Q_i)$ be two effective divisors such that $\supp(E_i)\cap\supp(D_j)=\emptyset$ for $i,j\in\{1,2\}$. 
\begin{cor}\label{GreenForm}
 We have 
 \begin{align*}
&\langle D_1-D_2,E_1-E_2\rangle_v\\
  =&-\log\prod^d_{i=1}\frac{|\theta_{a,b}(z(\iota(P_i))-z(S(E_1)))\theta_{a,b}(z(\iota(Q_i))-z(S(E_2)))|_v}{|\theta_{a,b}(z(\iota(P_i))-z(S(E_2)))\theta_{a,b}(z(\iota(Q_i))-z(S(E_1)))|_v}\\
  &-2\pi\sum^d_{i=1}\Im (z(S(E_1)-S(E_2)))^T\Im (\tau)^{-1}\Im (z(\iota(P_i))-z(\iota(Q_i))), 
 \end{align*} where for any $Q\in A$ the tuple $z(Q)\in\C^g$ satisfies $j(z(Q))=Q$.
\end{cor}
\begin{proof}
N\'eron functions are invariant  under translation of the divisor up to an additive
constant, see \cite[Theorem~11.2.1]{LangFund}.
But according to \cite[Theorem~5.5.8]{LangFund}, $[-1]^*(\Theta)$ is just $\Theta$ translated by $S(\mathfrak{K})$, where $\mathfrak{K}$ is a canonical divisor. Hence the desired result follows from Theorem \ref{ArakHril} and Proposition \ref{neronpazuki2}.
\end{proof}
\begin{rk}
In \cite{David} Holmes gives a more direct proof of Lemma \ref{GreenForm} using
\cite[\S13.6/7]{LangFund}, which relies on the theory of differentials of the third kind.
\end{rk}

We can use the previous result to compute intersections at archimedean places. In practice we need to be able to do the following:
\begin{itemize}
	\item[1)] Given $E\in \Div^0(C)$, find non-special $E_1,E_2$ such that $E=E_1-E_2$.
	\item[2)] Compute the period matrix $\tau$.
	\item[3)] Given $P_1\in C(\C)$ and $\tau$, determine $z\in \C^g$ such that $j(z)=\iota(P_1)$.
	\item[4)] Given $\tau$ and $z\in\C^g$, compute $\theta_{a,b}(z)=\theta_{a,b}(z,\tau)$. 
\end{itemize}

\section{The hyperelliptic case}\strut\\[1mm]\label{hyper}
We now discuss how the methods outlined in the previous section can be combined to give a practical algorithm for the computation of canonical heights in the case of hyperelliptic curves.

Suppose that $C$ is a hyperelliptic curve of genus $g$ defined over a field $k$, given as the smooth projective model of an equation
\begin{equation}\label{heq}
 Y^2+H(X,1)Y=F(X,1),
\end{equation}
where $F(X,Z),H(X,Z)\in k[X,Z]$ are forms of degrees $2g+2$ and $g+1$, respectively, and
the discriminant of the equation \eqref{heq} is nonzero. We will vary $k$ as in general discussion of the previous sections.

A different, but related approach to the computation of the local N\'eron symbols has been
developed independently by Holmes \cite{David}. 
The main difference lies in the computation of the non-archimedean intersection
multiplicities.
In \cite{David}, norm maps of non-archimedean field extensions are used instead of our
Gr\"obner basis approach.

\subsection{Finding suitable divisors of degree zero}\strut\\[1mm]\label{Move1h}
 Suppose that $D\in\Div(C)(k)$ has degree zero. Then the notions introduced in Section \ref{Move1} are all well-known: The reduction process is part of Cantor's algorithm for the addition of divisor classes introduced in \cite{Cantor}; here the divisor $A$ used for reduction is equal to $(\infty)$ when we have a $k$-rational Weierstrass point $\infty$ at infinity and is equal to $(\infty_1)+(\infty_2)$ when there are two branches $\infty_1,\infty_2$ over the singular point at infinity in the projective closure of equation \eqref{heq}.

In the former case Lemma \ref{RedEx} says that the reduction process yields the unique effective divisor $\tilde{D}$ such that \[D\sim\tilde{D}+r(\infty),\] where $0\le-r=\deg{\tilde{D}}\le g$ and $\deg(\tilde{D})$ is minimal. In the latter case it turns out that when $g$ is even we can still find a unique $\tilde{D}$ of minimal nonnegative even degree $-r\le g$ such that \[D\sim\tilde{D}+\frac r2((\infty_1)+(\infty_2))\] if we impose further conditions on its ideal representation. Conversely, if $g$ is odd we might have to take reductions of degree $g+1$ into account and these are not unique. However, uniqueness of the reduction is not an essential property in our applications and so we shall not discuss it any further. 

A possible ideal representation of a reduced effective divisor $D$ is given by the {\em Mumford representation} which we now recall briefly. 

If we view $C$ as embedded in weighted projective space with weights $1,g+1,1$ assigned to the variables $X,Y,Z$, then it is given by the equation
\[
Y^2+H(X,Z)Y=F(X,Z).
\]
An effective divisor $D$ of degree $d\le g+1$ corresponds to a pair of homogeneous forms $(A(X,Z),B(X,Z))$, where $A(X,Z)$ and $B(X,Z)$ have degrees $d$ and $g+1$ respectively, such that $D$ is defined by \[A(X,Z)=0=Y-B(X,Z)\] and we impose the additional condition that
\[A(X,Z)\;|\; B(X,Z)^2+H(X,Z)B(X,Z)-F(X,Z).\]

First suppose that there is a unique Weierstrass point $\infty$ at infinity in $C(k)$. Then any nonzero effective divisor $D=\sum^d_{j=1}(P_j)$ that is reduced along $(\infty)$ has degree $d\le g$ and cannot contain $\infty$ in its support. Hence we can safely dehomogenize in order to represent $D$ and so we may take
\[
I_D=(a(x),y-b(x)),
\]
where $a(x)=A(x,1)$ and $b(x)=B(x,1)$, for its ideal representation. More concretely, we have
 \[a(x)=\prod^n_{j=1}(x-x(P_j))\] 
 and $b(x)$ has minimal degree such that
 \[
  b(x(P_j))=y(P_j)\text{  for }j=1,\ldots,d.
 \]

Conversely, suppose that there are two points $\infty_1,\infty_2$ at infinity. Suppose that $D$ is reduced along $(\infty_1)+(\infty_2)$. If $\supp(D)$ does not contain a point at infinity, then we can dehomogenize as before to find an affine representation. If this does not hold, say $\infty_1\in \supp(D)$, then necessarily $\infty_1,\infty_2\in C(k)$ and $\infty_2\notin\supp(D)$. This case is more subtle, because we cannot tell the multiplicity of $\infty_1$ in $D$ from its dehomogenized form. For our applications it suffices to treat the affine and the infinite part of $D$ separately. Hence this complication does not cause any trouble.

Now let $k$ be a global field, and let the divisor $D_\infty$ be defined by $2(\infty)$ if there is a unique $k$-rational point at infinity and by $(\infty_1)+(\infty_2)$ otherwise. Also suppose $d$ is even and \[D=\tilde{D}-\frac d2D_\infty,\] where $\tilde{D}=\sum^d_{i=1}(P_i)$ is reduced along $D_{\infty}$ such that no $P_i$ is a point at infinity or a Weierstrass point. Then we can always use $n_1=1$ and $n_2=-1$ in the method introduced in Section \ref{Move1}; this is due to Holmes, see \cite{David}. Namely, if we apply the hyperelliptic involution 
\[
 Q\mapsto Q^-
\]
 to the points $P_i$, then we have
 \[
 D'=\sum^d_{i=1}(P^-_i)-\frac d2D_\infty\sim-D.
 \]
If we move this by the divisor of a function $x-\zeta$, where $\zeta\in k$ is such that $x(P_i)\ne \zeta$ for all $P_i$, then we find \[\supp(D)\cap\supp(E)=\emptyset,\] where $E=D'+d/2\div(x-\zeta)$. 
This corresponds to choosing $A=D_\infty$ and $A'=D(\zeta)$ in the method outlined above, where $D(\zeta)=\div(x-\zeta)+D_\infty$.
 
Instead of computing $\langle D,D\rangle$, we can now compute 
\[
 \hat{h}(P)=-\langle D,D\rangle=\langle D,-D\rangle=\langle D,E\rangle.
\]

If we have 
\[
D=\tilde{D}-\frac d2D_\infty,
\]
where
\[
 \tilde{D}=\sum_{i=1}^{d'} (P_i)+n_\infty(\infty_1)
\]
is reduced along $D_\infty=(\infty_1)+(\infty_2)$, such that $d=d'+n_\infty$  and all $P_i$ are affine non-Weierstrass points (see Section \ref{Move1h}), then we also have to move $D$ away from $\infty_1$ using a function $x-\zeta'$, where $x(P_i)\ne\zeta'\ne\zeta$ for all $i=1,\ldots,d'$. The computation becomes
\[
 -\langle D,D\rangle=\langle \sum_{i=1}^{d'} (P_i)+n_\infty(\infty_1)-\frac d2D({\zeta'}), \sum_{i=1}^{d'} (P^-_i)+n_\infty(\infty_2)-\frac d2D({\zeta})\rangle
\]
and poses no additional problems due to the bilinearity of the local N\'eron symbol.

What if there is a unique rational Weierstrass point $\infty$ at infinity and $d$ is odd? In that case we use
\[
 D'=2\sum^d_{i=1}(P^-_i)-dD_\infty\sim-2D
\]
and compute
\[
 \hat{h}(P)=-\langle D,D\rangle=\langle D,-D\rangle=\frac12\langle D,E\rangle,
\]
where $E=D'+d\div(x-\zeta)$ and $\zeta$ is as above. Note that we can still use the reduced Mumford representation, because we have
\[
\langle D,E\rangle=2\langle D,\sum^d_{i=1}(P^-_i)\rangle-d\langle D,D(\zeta)\rangle.
\]
Finally, if $\supp(D)$ contains an affine Weierstrass point, then we simply compute $\hat{h}(P)=\frac1{n^2}\hat{h}(nP)$ such that $nP$ has a reduced representation not containing an affine Weierstrass point.

\subsection{Determining relevant non-archimedean places}\strut\\[1mm]\label{factoringh}
Suppose that $k$ is a global field. Our curve $C$ is covered by two affine patches $C^1$ and $C^2$, where 
\begin{equation}\label{C1}
 C^1:y^2+H(x,1)y=F(x,1)
\end{equation}
and 
\begin{equation}\label{C2}
 C^2:w^2+H(1,z)w=F(1,z).
\end{equation}
It follows from the discussion at the end of Section \ref{factoring} that we can assume
$\O_k$ to be Euclidean. Suppose that $D$ and $E$ are
$I_{D,1}=(a(x),cy-b(x))$ and $I_{E,1}=(a'(x),c'y-b'(x))$ on $C^1$, respectively (where we have
multiplied all polynomials by the least common multiple of the denominators of their coefficients, if necessary). Then we need to compute a Gr\"obner basis of
\[
I_{D,E,1}=(y^2+H(x,1)y-F(x,1),a(x),a'(x),cy-b(x),c'y-b'(x))
\]
and factor the unique element $q_{D,E,1}$ of $\O_k$ appearing in this basis. 

Now suppose that $v\in M_k^0$ satisfies $i_v(D_{v,\KC},E_{v,\KC})>0$, where
$\KC\to\overline{C}^v$
is a desingularization in the strong sense and that the points of
intersection do not map to the closure of $C^1$ in $\KC$. 
Any such $v$ must satisfy $v(a_d)>0$ and $v(a'_{d'})>0$, where $a_d$ and $a'_{d'}$ are the leading coefficients of $a(x)$ and $a'(x)$, respectively. 
So instead of computing a Gr\"obner basis of $I_{D,E,2}$, we can factor $\gcd(a_d,a'_d)$
which is usually much easier than factoring $(q_{D,E,2})$. This simplification can make a big difference in practice. 

\subsection{Computing non-archimedean intersection multiplicities and the correction term}\strut\\[1mm]\label{nainth}
Let $k$ denote a non-archimedean local field and let $\pi$ be a uniformizing element.

Let  $D\in\Div(C)(k)$ be effective such that an ideal representation of $D$ on $C^1$ is 
\[
 I_{D,1}=(a(x),y-b(x)),
\]
where $a(x),b(x)\in k[x]$ and we have $\deg(a)\le g$ and $\deg(b)\le g+1$ as in Section \ref{Move1h}. 

In order to use Proposition \ref{IntForm} to compute non-archimedean intersection
multiplicities, we need to be able to decompose divisors into prime divisors over
unramified extensions. 
The main point distinguishing the hyperelliptic case from the general situation is that we
can decompose divisors by factoring univariate polynomials over non-archimedean local
fields; we show how this can be used in this section.

Note that factoring univariate polynomials over $p$-adic fields and Laurent series over finite fields is implemented in {\tt Magma} (following work of Pauli \cite{Pauli}). 

We first deal with the case $\KC=\overline{C}$ and use the affine cover $\KC=\KC^1\cup\KC^2$,
where $\KC^i$ is the Zariski closure over $S$ of the affine curve $C^i$ defined as in Section
\ref{factoringh}.

We can factor $a(x)=a_1(x)a_2(x)$, where $a_2(x)$ is constant modulo $\pi$ and $a_1(x)\in R[x]$. This corresponds to a decomposition $D=D_1+D_2$, where $D_{1,\KC}$ lies in $\KC^1$ and $D_{2,\KC}$ lies in $\KC^2$. More precisely, we have 
\[
I_{D_1,1}=(a_1(x),y-b_1(x)),
\]
where $b_1(x)=b(x) \bmod{a_1(x)}$. In order to use Proposition \ref{IntForm}, we need
$b_1(x)\in R[x]$. Suppose that $a_1(x)$ is irreducible (otherwise factor $a_1(x)$ into
irreducibles) and that $b_1(x)\notin R[x]$. If $D_{1,\KC}$ does not have a singular point
of the special fiber $\overline{C}_v$ in its support (for instance, if $a_1$ is
unramified), then we can safely extend $k$ by a root of $a_1$ and work over this
extension. Repeating this process, if necessary, leads to a finite extension $k'$ of $k$
such that $D_{1,\KC}$ splits into prime divisors over $k'$ that have $R'$-rational ideal
representations, where $R'$ is the ring of integers of $k'$.

Now suppose that $a_1(x)$ reduces to $(x-a)^m$ mod $\pi$, where $a$ is the $x$-coordinate
of a singular point of $\overline{C}_v$ and $m\ge1$. In general we cannot extend the field by a
root of $a_1$, because there may be points in the support of $D_{1,\KC}$ that are not
regular over this extension. But because of the special shape of $a_1$, we can simply use
the $R$-rational ideal representation $(a_1(x), \pi^sy-b'_1(x))$, where
$b_1(x)=\pi^{-s}b'_1(x)$,  and $b_1(x)\in R[x]$ has a unit among its coefficients.
Note that this approach does not always work for more general $a_1(x)$.

If we have $\KC\ne\overline{C}$, then we simply start by factoring $a(x)$ into irreducibles over
$k^{\mathrm{nr}}$. 
Assuming that $a_1(x)$ is one of the irreducible factors, we lift the ideal representation $I_{D_1,\KC}=(a_1(x), \pi^sy-b'_1(x))$ recursively through suitable blow-ups until we arrive at a suitable affine patch where the intersection multiplicities can be computed using Proposition \ref{IntForm}. 
As explained in Subsection \ref{CompCorrTerm}, this is also sufficient to compute the correction term.

\subsection{Computing archimedean intersection multiplicities}\strut\\[1mm]\label{ArchimInth}
In order to compute archimedean intersection multiplicities, we need algorithms for steps 1)--4) introduced at
the end of Section \ref{ArchimInt}.

For hyperelliptic curves, steps~2),~3) and~4) have all been implemented in {\tt Magma} by
van Wamelen. 
An earlier version of the implementation using {\tt Mathematica} can be found on van
Wamelen's homepage \cite{vanWamelen}. The routines there only work for genus~2 curves, but
the general algorithms in {\tt Magma} work similarly.

It is important to note that step 4), the computation of $\theta_{a,b}$, is done via approximation using
the definition, in particular it can be used to compute $\theta_{a,b}$ provably up to
desired precision.

We discuss step 1). Here we want to find, given $P,Q\in A$, divisors $D_1,\; D_2,\;E_1$ and $E_2$ such that
\begin{itemize}
	\item[(a)] $[D_1-D_2]=P$ and $[E_1-E_2]=Q$,
	\item[(b)] $D_1,\;D_2,\;E_1,\;E_2$ are effective and have pairwise disjoint support,
	\item[(c)] $E_1$ and $E_2$ are non-special.
\end{itemize}
We can allow ourselves more freedom, and only require that (a) holds for some multiple
$nQ$, due to the bilinearity of the local N\'eron symbol. For simplicity we only discuss
the case of a unique point at infinity, the other case being similar with a few minor subtleties if $g$ is odd. We pick $D_1:=\tilde{D}$ and $D_2:=d(\infty)$ if $P$ is represented by $\tilde{D}-d(\infty)$ and $\tilde{D}$ has affine support. Suppose that $nQ$ is represented by $\tilde{E_n}-g(\infty)$, where $\tilde{E_n}$ is non-special and has affine support such that $\tilde{E_n}$ and $\tilde{E_n'}$ have support disjoint from $\tilde{D}$, where $\tilde{E'_n}$ is the result of the hyperelliptic involution applied to the points in the support of $\tilde{E_n}$. Then $2nQ$ is represented by $\tilde{E_n}-\tilde{E_n'}$ and we choose $E_1:=\tilde{E_n}$ and $E_2:=\tilde{E_n'}$. .

With these choices, at most $d+g$ applications of the Abel-Jacobi map and at most $2d$ applications of the theta-function $\theta_{a,b}$ are required in order to compute $\langle D_1-D_2,E_1-E_2\rangle_v$ for an archimedean place $v$, essentially because we have $\iota(\infty)=0$.

Now let $\zeta\in k^*$ be as in Section \ref{Move1h}. We are actually interested in computing 
\begin{equation}\label{ulteq}
\langle \tilde{D}-d(\infty), \tilde{E}-e/2D(\zeta) \rangle_v,
\end{equation}so we compute a function $\beta\in k(C)^*$ such that
\[
\div(\beta)=E_1-E_2-2n\tilde{E}+ndD(\zeta)
\]
See \cite{Hess} for an algorithm that computes $\beta$. Using properties (a) and (c) of Proposition \ref{locneronprops} we can compute \eqref{ulteq}.

Notice that in contrast to the non-archimedean case the running times of steps 3) and 4) do not crucially depend on the heights of the points in the supports of the respective divisors, since we work with the complex uniformization.

\section{Examples}\strut\\[1mm]\label{exs}
In this section we provide a hyperelliptic example of a regulator that was computed using the algorithm outlined in the previous sections. Moreover, we shall discuss, at least in the case of hyperelliptic curves, how the running time changes as we increase
\begin{itemize}
\item[(a)] the genus of the curve;
\item[(b)] the size of the coefficients of the point.
\end{itemize}
All times are in seconds, unless noted otherwise. 
For the computations we have used~50 digits of $p$-adic and~30 digits of real and complex precision.

We first use the {\tt Magma}-implementation of our algorithm to compute the regulator of the Jacobian of a hyperelliptic genus~3 curve up to an integral square. 
\begin{ex}
Let $C$ be given by the smooth projective model of the equation
\[
Y^2=X(X-1)(X-2)(X-3)(X-6)(X-8)(X+8).
\]
The curve $C$ is a hyperelliptic curve of genus~3, defined over $\Q$. A quick search reveals the following rational non-Weierstrass points on $C$.
\[
(-2,\pm 240),(4,\pm48),(-6,\pm 1008)
\]
Let $A$ denote the Jacobian of $C$; obviously its entire 2-torsion subgroup is defined over $\Q$. In order to bound the Mordell-Weil rank of $A$ we compute the dimension of the 2-Selmer group of $A$ over $\Q$ using {\tt Magma}. This dimension is equal to~3 and hence we get an upper bound of~3 on the rank.  If $P_1,\ldots,P_n\in A$, then we denote the {\em regulator} of $P_1,\ldots,P_n$ by
\[
\Reg(P_1,\ldots,P_n)=\det \left((P_i,P_j)_{\mathrm{NT}}\right)_{i,j}.
\]

We want to compute the regulator $\Reg(P,Q,R)$ of the subgroup $G$ of $A(\Q)$ generated by the points 
\begin{eqnarray*}
P&=&(-2,-240)-(\infty)\\
Q&=&(4,-48)-(\infty)\\
R&=&(-6,1008)-(\infty).
\end{eqnarray*}

The discriminant of $C$ factors as $2^{50}3^{12}5^67^411^2$. We first find regular models
at the bad primes~2,~3,~5,~7 and~11. All computations in this example were done using {\tt
Magma} on a 1.73 GHz Pentium processor. It turns out that all computed regular models are
already minimal; we list the number of components of the special fiber of the respective
regular model, the (geometric) group of components $\Psi_p$ of the N\'eron model  and the time it took to compute the regular model in Table \ref{TabRegMod}.

\begin{table}\begin{tabular}{|c|c|c|c|}

\hline

prime & \# of comps.& $\Psi_p$ & time \\

\hline

2 & 14& $(\Z/2\Z)^5$ & 1.95 \\

3 &9 &  $(\Z/2\Z)^3\times\Z/4\Z$& 0.35 \\

5 &4 & $(\Z/2\Z)^3$& 0.23\\

7 & 3&$(\Z/2\Z)^2$& 0.29\\

11 &2&$\Z/2\Z$& 0.10\\
\hline

\end{tabular}\caption{Regular model data}\label{TabRegMod}\end{table}

After this preparatory step we now compute the entries of the height pairing matrix. The results and timings can be found in Table \ref{TabCanHts}, 
\begin{table}\begin{tabular}{|c|c|c|}

\hline

$S\in A(\Q)$ & $\hat{h}(S)$&  time \\

\hline

$P$ & 1.90008707521104082692048090266& 23.10  \\

$Q$ & 1.15261793630905629106514447088 & 19.76  \\

$R$ & 2.90090831616336727010940214290& 20.96 \\

$P+Q$&2.36481584203715381857836835238& 19.95 \\

$P+R$& 5.51584078564985349844572029952&20.67 \\
$Q+R$& 5.74901893484137170755580219303&21.22 \\
\hline\end{tabular}\caption{Canonical height computations}\label{TabCanHts}\end{table}

Using these results, we find
\[
\Reg(P,Q,R)=4.28880986177463283058861934366.
\]
We can test our findings by computing $\Reg(nP,mQ,lR)$ for several integral values of $n,m,l$. In all cases we get the relation
\[
\Reg(nP,mQ,lR)/\Reg(P,Q,R)=n^2m^2l^2
\]
up to an error of less than $10^{-29}$. 
As $\Reg(P,Q,R)$ is clearly non-zero, we know that $G$ is a subgroup of finite index and hence $\Reg(P,Q,R)$ equals $\Reg(A/\Q)$ up to a rational square. 

\end{ex}
Next we want to illustrate the behavior of the running time of our algorithm. We have refrained from a formal complexity analysis, mostly because the algorithm uses several external subroutines, such as the computation of regular models and of theta functions, whose complexities have not yet been analyzed. 

In the case of zero-dimensional ideals of polynomial rings over fields, the complexity of
a Gr\"obner basis computation can be shown to be polynomial in $D^n$, where $D$ is the
maximal degree of the elements of the basis we start with and $n$ is the number of
variables. See \cite{GroebCompl} for a summary of results regarding complexity of
Gr\"obner basis computations. In particular this holds for Faug\`ere's $F4$-algorithm
\cite{Faugere}, used for instance by {\tt Magma} (over fields and Euclidean rings). This
result can be extended easily to the case of polynomial rings over Euclidean domains,
provided we have fast algorithms available for the linear algebra computations in the
$F4$-algorithm, such as those implemented in {\tt Magma}. So the Gr\"obner basis
computations do not cause any trouble in practice, since the way regular models are
computed in {\tt Magma} ensures that the number of variables does not grow.

Indeed, the running time of the algorithm is usually dominated by the various analytic
computations required for the archimedean local N\'eron symbols.
They depend exponentially  on the genus; the  curve of largest genus that we have been able to compute with has genus~10, see Example \ref{ExCanFam} below. 
If the genus is not very large, but the size of the coefficients of the point $P\in A(k)$ that we want to compute the canonical height of is, then it turns out that the main bottlenecks are usually the factorizations alluded to in Section \ref{factoring}; recall that these are required in order to find out which places can lead to non-trivial non-archimedean local N\'eron symbols. 
See Example \ref{ExCanMult}. 
\begin{table}[!ht]\begin{tabular}{|c|c|c|c|c|}
\hline
$d$ &genus& $\hat{h}(P)$&act & nact \\
\hline
5 & 2& 1.20910894883943045491548486513 &3.51&0.33\\
7 & 3&  1.31935353209873515158774224282&6.70&0.34\\
9& 4&  1.39237255678179422540594853290&12.65&0.87\\
11& 5& 1.44187308116714103129667604112 &32.30&1.67\\
13&6& 1.47679608841931245229396457463 &120.51&2.99\\
15&7& 1.50265701979128671544005708236 &791.14&5.17\\
17&8& 1.52254076352483838532148827258 &4729.03&8.95\\
19&9&1.53829882683402848666502818888 &62535.55&14.20\\
21&10&1.55109127084768378637549292754&280731.59&21.35\\
\hline
\end{tabular}\caption{Canonical heights in a family}\label{TabCanFam}\end{table}

 \begin{table}[!ht]\begin{tabular}{|c|c|c|c|c|c|}
\hline
$n$ &$\hat{h}(nP)$& act & nact&factt&digits\\
\hline
1&0.19809838973401855248161508134&2.35&0.02&0.00&1\\
2&0.79239355893607420992646032538&2.38&0.02&0.00&1\\
3&1.78288550760616697233453573211&4.06&0.13&0.00&1\\
4&3.16957423574429683970584130153&3.36&0.11&0.00&1\\
5&4.95245974335046381204037703364&2.91&0.10&0.00&1\\
6&7.13154203042466788933814292846&3.39&0.08&0.00&3\\
7&9.70682109696690907159913898594&3.40&0.06&0.00&6\\
8&12.6782969429771873588233652061&3.36&0.34&0.05&9\\
9&16.0459695684555027510108215890&3.31&0.29&0.01&11\\
10&19.8098389734018552481615081345&3.41&0.95&0.64&16\\
11&23.9699051578162448502754248428&3.33&0.37&0.08&19\\
12&28.5261681216986715573525717137&3.45&0.47&0.11&21\\
13&33.4786278650491353693929487474&3.32&0.34&0.09&21\\
14&38.8272843878676362863965559437&3.42&196.87&196.50&30\\
15&44.5721376901541743083633933028&3.37&0.53&0.20&38\\
16&50.7131877719087494352934608245&3.39&0.90&0.25&42\\
\hline
\end{tabular}\caption{Canonical heights for multiples of a point}\label{TabCanMult}\end{table}

All computations for the following two examples were done using a 3.00 GHz Xeon processor.
\begin{ex}\label{ExCanFam}
Consider the family \[C_d:y^2=x^d+3x^2+1\] for $d\in\{5,7,9,11,13,15,17,19,21\}$ and let
$P=[(0,1)-(0,-1)]\in A_d(\Q)$, where $A_d$ is the Jacobian of $C_d$. We compute
$\hat{h}(P)$ and record the running time for both the archimedean and the non-archimedean computations. See Table \ref{TabCanFam}, where nact and act denote non-archimedean and archimedean computation time, respectively. This example illustrates the exponential dependency on the genus.
\end{ex}
\begin{ex}\label{ExCanMult}
	Next we consider $C:y^2=x^{10}-x^3+1$ and let $P\in A(\Q)=[
	(0,1)-\infty_+]\in\mathrm{Jac}(C)(\Q)$, where $\infty^+$ is the point at infinity such
	that the function $y/x^3$ has value~1 at $\infty_+$.
	The curve $C$ has bad reduction at~2.
	We use {\tt Magma} to compute a regular model at~2; this takes~1.97 seconds and
	yields~9 irreducible components.


We compute the canonical heights of positive multiples of $P\in A(\Q)$ and record running times. The results are in
Table \ref{TabCanMult} and we see that we have $\hat{h}(nP)=n^2\hat{h}(P)$ for all
$n\in\{1,\ldots,16\}$. Here nact and act have the same meaning as in Table
\ref{TabCanFam}, factt denotes the time spent on integer factorization and digits denotes
the number of digits of the maximal height of the coefficients of the polynomials in the
Mumford representation of $nP$.

For $n\ge17$, the time spent on factoring increases dramatically. 
For instance, the factorization needed for $n=18$ takes about~62 hours.
The largest multiple for which we are able to compute $\hat{h}(nP)$ is $n=21$, where
digits$=79$. 

\end{ex}

\section{Outlook}\strut\\[1mm]\label{outl}
It is now possible, using the {\tt Magma}-implementation of the algorithm described in
this work (see \cite{Homepage}), to compute canonical heights on Jacobians of hyperelliptic curves defined over number fields. There is work in progress on most of the applications outlined in the introduction. Some can now be tackled in a straightforward way, such as the computation of regulators up to integral squares, which can be used to gather numerical evidence for the conjecture of Birch and Swinnerton-Dyer as in \cite{FLSSSW}, some require more work to be done first, such as the computation of generators of the Mordell-Weil group. An algorithm for the latter is presented in \cite{StollH2}, but in order to apply it, one also needs a suitable naive height combining the properties that we can list all points of naive height up to some bound and that the difference between the two heights can be bounded effectively. 
Holmes has recently come up with a good candidate for such a naive height, the details
will appear in his upcoming PhD thesis at the University of Warwick. 
For the genus three case, there is recent work of Stoll \cite{StollG3} solving this
problem.

We now sketch some possible directions for further research regarding the canonical height algorithm itself. First, our algorithm works for any global field and hence it should not be too difficult to implement it for hyperelliptic curves defined over global function fields. In fact, some problems disappear because of the absence of archimedean places. More importantly, it would be interesting to extend our algorithm to the case of non-hyperelliptic curves. Here, there are essentially two problems:
\begin{itemize}
\item[(i)] How can we decompose divisors into prime divisors? (see Section \ref{factoring})
\item[(ii)] How can we implement the analytic steps 2) -- 4) introduced at the end of Section \ref{ArchimInt}?
\end{itemize}
There are 3 approaches to problem (i). If we could factor multivariate polynomials over non-archimedean local fields, then (i) would be solved, but such an algorithm has not been implemented to the author's knowledge. In favorable situations it may be possible to use ideal representations similar to the Mumford representation of hyperelliptic curves and thus decompose divisors using factorization of univariate polynomials. An ideal representation resembling Mumford representation has been proposed in \cite{spq} for smooth plane quartics. Finally, it might not be necessary to decompose divisors at all, if we could make the approach mentioned in Remark \ref{NoFac} and described in \cite{HessStudent} work in our situation.

Of course, problem (i) disappears whenever we deal with divisors having pointwise $k_v$-rational support for each relevant non-archimedean local field $k_v$. We have used this to compute all non-archimedean local N\'eron symbols necessary for the computation of the regulator (up to an integral square) of the Jacobian of a non-hyperelliptic curve of genus~4 without special properties, see \cite[Chapter~6]{thesis}. This curve plays a major part in \cite{XDyn}, where it is shown, assuming the first part of the conjecture of Birch and Swinnerton-Dyer, that there are no rational cycles of length~6, an important result in arithmetic dynamics. Our goal was to verify the second part of the conjecture of Birch and Swinnerton-Dyer up to an integral square for this particular curve, a challenge problem posed by Stoll in \cite{XDyn}. 

Regarding problem (ii), an extension of van Wamelen's algorithms to the non-hyperelliptic case would suffice. 
This does not appear to be particularly difficult, but we have not attempted it. 
The main obstacle is the choice of a basis of holomorphic differentials, which will
probably depend on the class of curves under consideration (for hyperelliptic curves there
is, of course, a canonical choice and this is used by van Wamelen's algorithms).

If we are willing to allow non-rigorous numerical integration methods, then all of the 
relevant algorithms have been developed (see \cite{CompPeriod, CompTheta, CompAbel}) by Deconinck et al. for general compact Riemann surfaces. 
However, these algorithms are not suitable in order to actually {\em prove} that we have computed the
correct canonical height up to given precision.

In any case, Deconinck and his collaborators implemented their algorithms in {\tt Maple} in a
package called {\em algcurves}.
Unfortunately, the {\tt Maple} developers have since decided to change some of the functions that algcurves uses, in the process destroying some of the package's crucial functionality. 
Deconinck's group are currently working on a long-term project to rewrite all necessary routines in {\tt Sage}. 

Finally, it would be interesting to formally analyze the complexity of our algorithm. 
While knowing that the dependence on the genus is exponential due to the necessity of computing theta-functions, we first need to analyze the complexity of {\tt Magma}'s desingularization algorithm and the analytic algorithms that we use before more can be said.

\begin{rk}
Let $A$ be the Jacobian of a smooth projective curve $C$ defined over a number field and let
$p$ be a prime such that $A$ has good ordinary reduction at all $v\mid p$.
Coleman and Gross \cite{ColemanGross} have constructed a pairing taking values in $\Q_p$ between divisors $D,E$ on
$C$ of degree~0 which can be decomposed into a sum of local height pairings $h_v(D,E)$ over all non-archimedean $v\in M_k$. 

They show that this pairing respects linear equivalence and coincides with the $p$-adic height
pairings on $A$ constructed by Schneider \cite{Schneider} and Mazur-Tate \cite{MT}.


If $v\mid p$, then $h_v(D,E)$ can be expressed in terms of Coleman integration. An
algorithm for the computation of $h_v(D,E)$ was introduced by Balakrishnan and Besser in
\cite{BalaBess}, see also \cite[Chapter~8]{JenThesis}. 
Computing the local height pairings at $v\nmid p$ is equivalent to computing the local N\'eron symbol at $v$ (cf.
\cite[\S2]{ColemanGross}). Hence we can combine the results presented in this work with the method of Balakrishnan and Besser, leading to the first algorithm to compute $p$-adic heights for $g\ge2$.

One can define $p$-adic regulators similar to the classical case (see Section \ref{exs}).
Together with Balakrishnan we have computed the $p$-adic regulator for all but one of the modular abelian surfaces considered in \cite{FLSSSW} for all good ordinary $p<100$. This gives rise to an extension of the $p$-adic Birch and Swinnerton-Dyer conjecture for elliptic curves due to Mazur, Tate and Teitelbaum to the case of modular abelian surfaces (cf. \cite[Conjecture~9.1.4]{JenThesis}).  
\end{rk}


\end{document}